\documentclass[10pt]{amsart}

\usepackage[headings]{fullpage}

\usepackage{lmodern}
\usepackage[T1]{fontenc}

\usepackage{amsmath}
\usepackage{amsthm}
\usepackage{amssymb}
\usepackage{amsfonts}
\usepackage{latexsym}

\usepackage{paralist}

\usepackage{color}

\usepackage[colorlinks,pagebackref=true,pdftex]{hyperref}
\usepackage[alphabetic,abbrev,backrefs,lite]{amsrefs}

\makeatletter

\@addtoreset{equation}{section}
\def\theequation{\thesection.\@arabic \c@equation}
\def\@citecolor{blue}
\def\@urlcolor{blue}
\def\@linkcolor{blue}
\def\theenumi{\@roman\c@enumi}

\makeatother

\theoremstyle{plain}
\newtheorem{theorem}[equation]{Theorem}
\newtheorem{lemma}[equation]{Lemma}
\newtheorem{corollary}[equation]{Corollary}
\newtheorem{proposition}[equation]{Proposition}

\theoremstyle{definition}

\newtheorem{remark}[equation]{Remark}
\newenvironment{remarkbox}[1][]{%
    \begin{remark}[#1] \pushQED{\qed}}{\popQED \end{remark}}

\newtheorem{discussion}[equation]{Discussion}
\newenvironment{discussionbox}[1][]{%
    \begin{discussion}[#1]\pushQED{\qed}}{\popQED \end{discussion}}
\newtheorem{observation}[equation]{Observation}
\newenvironment{observationbox}[1][]{%
    \begin{observation}[#1]\pushQED{\qed}}{\popQED \end{observation}}

\newcommand{\define}[1]{\emph{#1}}

\newcommand{\naturals}{\mathbb{N}}
\newcommand{\ints}{\mathbb{Z}}
\newcommand{\complex}{\mathbb{C}}

\DeclareMathOperator{\In}{in}
\DeclareMathOperator{\charact}{char}
\DeclareMathOperator{\hilbertPoset}{\mathcal H}
\DeclareMathOperator{\idealPoset}{\mathcal I}
\DeclareMathOperator{\tor}{Tor}
\DeclareMathOperator{\homology}{H}
\DeclareMathOperator{\coker}{coker}
\DeclareMathOperator{\image}{Im}
\DeclareMathOperator{\rank}{rank}

\def\hssymb{\mathfrak z}

\newcommand{\fraka}{{\mathfrak a}}
\newcommand{\frakb}{{\mathfrak b}}
\newcommand{\calB}{{\mathcal B}}
\newcommand{\bfe}{{\mathbf e}}
\newcommand{\bff}{{\mathbf f}}
\newcommand{\frakI}{{\mathfrak I}}
\newcommand{\frakm}{{\mathfrak m}}

\renewcommand{\to}{\longrightarrow}
\DeclareMathOperator{\lpp}{lpp}

\title
{Poset Embeddings of Hilbert functions and Betti numbers}
\author[G.~Caviglia]{Giulio Caviglia}
\address{Department of Mathematics, Purdue University, West Lafayette IN
47907, USA}
\email{gcavigli@math.purdue.edu}

\author[M.~Kummini]{Manoj Kummini}
\address{Chennai Mathematical Institute, Siruseri, Tamilnadu 603103, India}
\email{mkummini@cmi.ac.in}
\thanks{The work of the first author was supported by a grant from the
Simons Foundation (209661 to G.~C.). 
In addition, both the authors thank 
Mathematical Sciences Research Institute, Berkeley CA,
where part of this work was done,
for support and hospitality during Fall 2012.}

\begin{document}

\begin{abstract}
We study inequalities between graded Betti numbers of 
ideals in a standard graded algebra over a field and their images under 
embedding maps, defined earlier by us in [Math.~Z. \textbf{274} (2013), no.~3-4,
pp.~809-819; arXiv:1009.4488].
We show that if graded Betti numbers 
do not decrease when we replace ideals in an algebra 
by their embedded versions, then the same behaviour is carried over 
to ring extensions.
As a corollary we give alternative inductive proofs of
earlier results of Bigatti, Hulett, Pardue, Mermin-Peeva-Stillman and
Murai.
We extend a hypersurface restriction theorem of Herzog-Popescu to the
situation of embeddings.
We show that we can obtain the Betti table of an ideal in the extension
ring from the Betti table of its embedded version
by a sequence of consecutive cancellations.
We further show that the lex-plus-powers
conjecture of Evans reduces to the Artinian situation.  
\end{abstract}

\maketitle

\section{Introduction}
\label{sec:intro}

This paper is part of our study of embeddings of the poset of Hilbert
series into the poset of ideals, in a standard-graded algebra $R$ over a
field $\Bbbk$. 
In our earlier paper~\cite{CaKuEmbeddings10}, we looked into
ways of embedding the set $\hilbertPoset_R$ of all Hilbert  series of
graded $R$-ideals (partially ordered by comparing the coefficients) 
into the set $\idealPoset_R$ of graded $R$-ideals (partially ordered by
inclusion).
Here we look at the behaviour of graded Betti numbers when we replace an
ideal by an ideal (with the same Hilbert function) that is in the image of
the embedding.

Examples of such embeddings, studied classically, are polynomial
rings and quotient rings of polynomial rings by regular sequences 
generated by powers of the variables.
If $R$ is a polynomial ring, then for every
$R$-ideal $I$, there exists a lex-segment ideal $L$ in $R$ such that the
Hilbert functions of $I$ and $L$ are identical (a theorem of
F.~S.~Macaulay, see~\cite {BrHe:CM}*{Section~4.2})
and, moreover, each of the graded Betti numbers of $L$ is at least as large
as the corresponding graded Betti numbers of $I$ 
(A.~M.~Bigatti~\cite {BigaUpperBds93}, 
H.~A.~Hulett~\cite {HuleMaxBettiNos93} and 
K.~Pardue~\cite {PardueDefClass96}).
Similarly, if $R = \Bbbk[x_1, \ldots, x_n]/(x_1^{e_1}, \ldots, x_n^{e_n})$
for some integers $2 \leq e_1 \leq \cdots \leq e_n$, then for every
$R$-ideal $I$ there exists a lex-segment 
$\Bbbk[x_1, \ldots, x_n]$-ideal $L$ such that $I$ and $LR$ have 
identical Hilbert functions (J.~B.~Kruskal~\cite{KruskalNoOfSimpl63} 
and G.~Katona~\cite{KatonaFiniteSets68} 
for the case $e_1=\cdots=e_n=2$,
and G.~F.~Clements and B.~Lindstr\"{o}m
\cite{ClemLindMacaulayThm69}, in general). Again, as 
$\Bbbk[x_1, \ldots, x_n]$-modules, each  graded Betti
number of $R/LR$ is at least as large as the corresponding graded Betti
number of $R/I$; this was proved by J.~Mermin, I.~Peeva and
M.~Stillman~\cite{MPSSquares08} for the case $e_1=\cdots=e_n=2$ 
in characteristic zero, by S.~Murai~\cite
{MuraiBPP08} in the general case for strongly
stable ideals and by Mermin and Murai~\cite {MeMuLPPPurePwrs08} in full
generality.  In both these cases, mapping the Hilbert series of $I$ to $L$
(or to $LR$ in the second case) gives an embedding of $\hilbertPoset_R$
into $\idealPoset_R$, such that the graded Betti numbers do not decrease
after the embedding.
See~\cites {GMPMappingCones11, GMPveronese10, IyPaMaxMinResln98}
for comparing the graded Betti numbers of $R/I$ and $R/LR$
over $\Bbbk[x_1, \ldots, x_n]$ and \cite{MuPeCLRings12} for 
comparing the graded Betti numbers over $R$.

Before we spell out what is done in this paper, let us fix some notation.
By $z$ we mean an indeterminate over $R$. Set $S = R[z]/(z^t)$, where $t
\geq 1 \in \naturals$ or $t=\infty$; if $t=\infty$, we mean that $z^t = 0$. By $A$
we denote a standard-graded polynomial ring over $\Bbbk$ that minimally
presents $R$, i.e., $\ker (A \to R)$ does not have any linear forms.
Let $B = A[z]$. The
\emph{graded Betti numbers} of a finitely generated graded $R$-module $M$
are $\beta_{i,j}^R(M) = \dim_\Bbbk \tor_i^R(M, \Bbbk)_j$. The \emph{Betti
table} of $M$, denoted $\beta^R(M)$, is 
$\sum_{i,j}\beta_{i,j}^R(M)\bfe_{i,j}$, where 
$\{\bfe_{i,j}, i \in \naturals, j \in \ints\}$ 
is the standard basis of $\ints^{\naturals \times \ints}$.

In this paper, we show that if graded Betti numbers 
do not decrease when we replace $R$-ideals by their embedded versions, 
then the same behavior extends to $S$-ideals; see
Theorem~\ref{theorem:bettiNumbers} for the precise statement.  
This
generalizes analogous results of Bigatti, Hulett, Pardue and Murai
mentioned above. Moreover, our technique gives short and self-contained
proofs of~\cite {MPSSquares08}*{Theorem~1.1} (more precisely, the key
intricate step for proving it for 
{strongly-stable-plus-squares}
ideals --- see~\cite {MPSSquares08}*{Section~3~and~4})
and~\cite
{MuraiBPP08}*{Theorem~1.1} (which generalizes the above theorem to
arbitrary powers); see Corollary~\ref{corollary:MPSMurai}.

We show that we can obtain {the Betti table of an $S$-ideal
from the Betti table of its embedded version}
by a sequence of consecutive cancellations. This is analogous to, and
motivated by, the similar result proved by Peeva
comparing the Betti table of an arbitrary homogeneous ideal in a polynomial
ring with the Betti table of the lex-segment ideal with the same Hilbert
function~\cite {PeevConsecCanc04}*{Theorem~1.1}. Additionally, we show that
a minimal graded $B$-free resolution of an ideal in the image of the
embedding map can be seen as an iterated mapping cone. The simplest
instance of this is the Eliahou-Kervaire resolution
(see~\cites{ElKeMinimalReslns90, PeStEKResln08}) of strongly-stable
monomial ideals in polynomial rings.
Similar
iterated mapping cones have been considered for ideals containing monomial
regular sequences, and have been used to obtain exact expressions for
graded Betti numbers; see~
\cites{MPSSquares08, MuraiBPP08, GMPMappingCones11}.
As an application of Theorem~\ref{theorem:bettiNumbers}, we show that the
lex-plus-powers conjecture of E.~G.~Evans reduces to the Artinian
situation (Theorem~\ref{theorem:lppReducesToArtinian}).

A word about the proofs. 
Problems on finding bounds for Betti numbers such as those
studied in much of the body of work above 
can be, inductively, reduced to studying 
polynomial extensions by one variable, or their quotients.
This is what we address in Theorem~\ref{theorem:bettiNumbers}.
Its proof makes crucial use  of~\cite{CaKuEmbeddings10}*{Theorem~3.10}
(reproduced as Theorem~\ref{theorem:extnRings}\eqref{theorem:strongHyp}
below) which is an analogue of the Hyperplane Restriction Theorem of
M.~Green~\cite {GreenHyperplane89}, and of similar results of J.~Herzog and
D.~Popescu~\cite {HerzPopeRestrictionTheorem98} and of Gasharov~\cite
{GasharovRestrictionTheorem99}.  Caviglia and
E.~Sbarra~\cite{CaSbLPPCoh12}*{Theorem~3.1}
use~Theorem~\ref{theorem:extnRings}\eqref{theorem:strongHyp},
again, to prove a
result similar to Theorem~\ref{theorem:bettiNumbers}, where graded Betti
numbers are replaced by Hilbert series of local cohomology modules.
Also see~\cite {ChPeGaMaxBettiNos02} for a proof of the
aforementioned result of Bigatti and Hulett, in characteristic zero, using
Green's Hyperplane Restriction Theorem.

We thank the referee for a careful reading and suggestions that improved
the exposition. The computer algebra system \texttt{Macaulay2}~\cite{M2}
provided valuable assistance in studying examples.

\section{Embeddings}
\label{sec:background}
We recall some notation and definitions from~\cite {CaKuEmbeddings10}.
Let $\Bbbk$ be a field and $A$ a finitely generated polynomial ring over
$\Bbbk$. We treat $A$ as standard graded, i.e., the indeterminates have
degree one. 
Let $\fraka$ be a homogeneous $A$-ideal and $R = A/\fraka$.
Let $\idealPoset_R = \{J : J \ \text{is a homogeneous $R$-ideal}\}$,
considered as a poset under inclusion. For a finitely generated
graded $R$-module $M = \oplus_{t \in \ints} M_t$,
the \define{Hilbert series} of $M$ is the formal power series
\[
H_M(\hssymb) = \sum_{t \in \ints} \left(\dim_\Bbbk M_t\right) \hssymb^t \in
\ints[[\hssymb]].
\]
The \define{poset of Hilbert series of homogeneous $R$-ideals} is the set
$\hilbertPoset_R = \{H_{J} : J \in \idealPoset_R\}$ endowed with the
following partial order: $H \succcurlyeq H'$ (or $H' \preccurlyeq H'$) 
if, for all $t \in \ints$,  the
coefficient of $\hssymb^t$ in $H$ is at least as large as that in $H'$.
(We point out that by $H_I$ we mean
the Hilbert series of $I$ and not of $R/I$.) 

In~\cite{CaKuEmbeddings10}  we studied the following question: 
given such a standard-graded $\Bbbk$-algebra $R$,
is there an (order-preserving) embedding $\epsilon : \hilbertPoset_R
\hookrightarrow \idealPoset_R$ as posets, such that $\mathrm{H} \circ \epsilon =
\mathrm{id}_{\hilbertPoset_R}$, where $\mathrm{H}:\idealPoset_R \longrightarrow
\hilbertPoset_R$ is the function $J \mapsto H_{J}$?
We will say that $\hilbertPoset_R$ \define{admits an embedding into}
$\idealPoset_R$ (and often, by abuse of terminology, merely that
$\hilbertPoset_R$ \define{admits an embedding})
if this question has an affirmative answer. 

Recall (from Section~\ref{sec:intro}) 
that $R = A/\fraka$, $B=A[z]$ and $S = R[z]/(z^t) = 
B/(\fraka B+(z^t))$  where $t$ is a positive integer or is
$\infty$.
Let $\frakI = \{i \in \naturals : i < t\}$.
Treat $B$ as \define{multigraded}, with $\deg x_i = (1,0)$ and $\deg z =
(0,1)$ and let the grading on $S$ be the one induced by this choice.
(In order to study embeddings of $\hilbertPoset_S$, we think of $S$ as
standard-graded, but we will use its multigraded structure, which is a
refinement of the standard grading, to construct them.)

Let $J$ be a multigraded $S$-ideal. It is isomorphic, as
an $R$-module, to $\bigoplus_{i\in \frakI}  J_{\langle i \rangle}z^i$, 
where for every $i\in \frakI$, $J_{\langle i \rangle} = (J:_Sz^i) \cap R$.
Notice that for all $i \in \frakI$, $i > 0$,
$J_{\langle i-1 \rangle} \subseteq J_{\langle i \rangle}$. We say that 
$J$ is \define{$z$-stable} if 
${\frakm}_RJ_{\langle i \rangle} \subseteq J_{\langle i-1 \rangle}$,
for all $i\in \frakI, i>0$ where $\frakm_R$ denotes the homogeneous maximal
ideal of $R$. We denote the set $\{H_J : J \;\text{is a $z$-stable
$S$-ideal}\}$ by $\hilbertPoset_S^{\mathrm{stab}}$. The main results from
our earlier work~\cite{CaKuEmbeddings10} are:

\begin{theorem}
\label{theorem:extnRings}
Let $t$ be a positive integer or $\infty$.
Let $S = R[z]/(z^t)$. 
Suppose that
$\hilbertPoset_R$ admits an embedding $\epsilon_R$. Then:
\begin{enumerate} 
\item 
\label{theorem:EmbExtendsToStableHilbPoset} 
There exists an embedding of posets 
$\epsilon_S: \hilbertPoset_S^{\mathrm{stab}} \hookrightarrow \idealPoset_S$ 
such that for all $H \in \hilbertPoset_S^{\mathrm{stab}}$, the Hilbert
function of $\epsilon(H)$ is $H$.  
\item \label{theorem:strongHyp}
If $J$ is a $z$-stable
$S$-ideal and $L = \epsilon_S(H_J)$, then
$H_{(J+(z^i))} \succcurlyeq H_{(L+(z^i))}$ for all $i \in \frakI$.
\end{enumerate}
\end{theorem} 

Part \eqref{theorem:EmbExtendsToStableHilbPoset} is the content
of~\cite{CaKuEmbeddings10}*{Theorem~3.3}, where we further assume that for
all $S$-ideals $J$, there is a $z$-stable $S$-ideal with the same Hilbert
function (i.e., $\hilbertPoset_S^{\mathrm{stab}} = \hilbertPoset_S$) and
conclude that $\hilbertPoset_S$ admits an embedding. 
If $t = \infty$, then 
$\hilbertPoset_S^{\mathrm{stab}} =
\hilbertPoset_S$~\cite{CaKuEmbeddings10}*{Lemma~4.1}. When $t$ is finite,
if $\Bbbk$ contains a primitive $t$th root of unity, 
$\fraka$ is a monomial ideal in some $\Bbbk$-basis $\calB$ of $A_1$,
$l^t \in \fraka$ for all $l \in \calB$ 
then, again, $\hilbertPoset_S^{\mathrm{stab}} =
\hilbertPoset_S$~\cite{CaKuEmbeddings10}*{Lemma~4.2}.
Part \eqref{theorem:strongHyp},  which can be viewed as a hypersurface
restriction theorem, is~\cite{CaKuEmbeddings10}*{Theorem~3.10};
see Discussion~\ref{discussionbox:hyperSurfRestr}.

\begin{remarkbox}
\label{remarkbox:EmbedIndividualPieces}
Let $J$ be a multigraded $S$-ideal. Write 
$J = \bigoplus_{i\in \frakI}  J_{\langle i \rangle}z^i$.  Then
$J' := \bigoplus_{i\in \frakI}  \epsilon_R(H_{J_{\langle i \rangle}})z^i$ 
is an $S$-ideal. If $J$ is additionally $z$-stable, then so is $J'$. To see
this, we note that,
since $J$ is $z$-stable, 
$H_{J_{\langle i-1 \rangle}}
\succcurlyeq
H_{{\frakm}_RJ_{\langle i \rangle}}$, and
therefore
$\epsilon_R(H_{J_{\langle i-1 \rangle}})
\supseteq
\epsilon_R(H_{{\frakm}_RJ_{\langle i \rangle}})
\supseteq
{\frakm}_R\epsilon_R(H_{J_{\langle i \rangle}})$,
where the last inclusion follows by using embedding filtrations~\cite
{CaKuEmbeddings10}*{Definition~2.3}. 
\end{remarkbox}

\begin{remarkbox}
\label{remark:propertiesEX} 
The map $\epsilon_S: \hilbertPoset_S^{\mathrm{stab}} \hookrightarrow
\idealPoset_S$ constructed in the proof of Theorem \ref{theorem:extnRings}
has additional properties inherited from $\epsilon_R$,
namely, if an $S$-ideal $L$ is in the image of
$\epsilon_S$, 
it is multigraded, $z$-stable and, when written
as $L=\bigoplus_{i\in \frakI}  L_{\langle i \rangle}z^i$,  the ideal 
$L_{\langle i \rangle}$ is in the image of $\epsilon_R$
for all $i \in \frakI$, 
i.e. $\epsilon_R(H_{L_{\langle i \rangle}})=L_{\langle i \rangle}$.
\end{remarkbox}

\begin{discussionbox}[\textbf{Restriction to general hypersurfaces}]
\label{discussionbox:hyperSurfRestr}
The Hyperplane Restriction Theorem of Green~\cite {GreenHyperplane89} 
(see, also,~\cite {GreenGins98}*{Theorem~3.4}) asserts that if $\Bbbk$ is
an infinite field,
$J$ is a $B$-ideal and $L$ the lex-segment $B$-ideal with
$H_I = H_L$, then for any general linear form $f \in S$,   
$H_{J + (f)} \succcurlyeq H_{L+(z)}$. 
(In the lexicographic order on $B$, $z$ is the last variable.)
Herzog and 
Popescu~\cite{HerzPopeRestrictionTheorem98}*{Theorem in Introduction} (in
characteristic zero) and 
Gasharov~\cite{GasharovRestrictionTheorem99}*{Theorem~2.2(2)} (in arbitrary
characteristic) generalized this to forms of arbitrary 
degree:
$H_{J + (f)} \succcurlyeq H_{L+(z^d)}$ for all general forms $f$ of degree $d$ for
all $d \geq 1$.
Restated in the language of embeddings, this is
$H_{J + (f)} \succcurlyeq H_{\epsilon_B(H_J)+(z^d)}$. Therefore, we may wonder
whether this is true more generally than for polynomial rings. More precisely, 
putting ourselves in the context of Theorem~\ref{theorem:extnRings},
we show that if $\charact \Bbbk = 0$, then
\begin{equation}
\label{equation:charZeroHRT}
H_{J + (f)} \succcurlyeq H_{\epsilon_S(H_J)+(z^d)}
\end{equation}
for all general
homogeneous elements $f \in S$ of degree $d$ and for all $d \in \frakI$. We
also show that the conclusion fails in positive characteristic.
In characteristic zero, the argument is as follows:
Let $g$ be a change  of
coordinates, fixing all the $x_i$ and sending $z$ to a sufficiently general
linear form. Let $w$ be the weight $w(x_i) = 1$ for all $i$ and $w(z) = 0$.
In characteristic zero, $\In_w(g\!\cdot\!J)$ is $z$-stable. (The proofs
of~\cite{CaKuEmbeddings10}*{Lemmas~4.1 and~4.2} use many steps of
distraction and taking initial ideals with respect to $w$, but
in characteristic zero, they can be replaced by a single step.) 
Now, $H_{J+(f)} = H_{g\!\cdot\!J+(g\!\cdot\!f)}$.
%
%
By
\cite{CavThesis04}*{7.1.2 and~7.1.3},
$H_{g\!\cdot\!J+(g\!\cdot\!f)} \succcurlyeq 
H_{\In_w(g\!\cdot\!J)+(g\!\cdot\!f)}$, 
%
Had we chosen $f$ to be $g^{-1}z^d$, we would have had the equality
$H_{\In_w(g\!\cdot\!J)+(g\!\cdot\!f)} = H_{\In_w(g\!\cdot\!J)+(z^d)}$;
however since $f$ is general, we have 
$H_{\In_w(g\!\cdot\!J)+(g\!\cdot\!f)} \succcurlyeq 
H_{\In_w(g\!\cdot\!J)+(z^d)}$. 
%
%
Finally, using
Theorem~\ref{theorem:extnRings}\eqref{theorem:strongHyp}, we see that
$H_{\In_w(gJ)+(z^d)} \succcurlyeq H_{\epsilon_S(H_{\In_w(gJ)})+(z^d)}$. Note
that $\epsilon_S(H_{\In_w(gJ)}) = \epsilon_S(H_J)$.

Now to show that the conclusion does not hold in positive characteristic,
consider $R = \Bbbk[x,y]/(x^p, y^p)$ where $p = \charact \Bbbk$. Let 
$S = R[z]$, $l$ a general linear form in $x,y$ and $J = z^pS$.
Then, $\epsilon_S$ denoting the embedding induced by the lexicographic
order on $\Bbbk[x,y,z]$, we have
\[
H_{J+(z+l)} \preccurlyeq H_{\epsilon_S(H_J)+(z)} \;\text{and}\;
H_{J+(z+l)} \neq H_{\epsilon_S(H_J)+(z)}; 
\]
contrast this with~\eqref{equation:charZeroHRT}.
To see
this, let us look at the corresponding quotients: $S/{J+(z+l)} \simeq R$
and, since $x^{p-1}y \in \epsilon_S(H_J)$, we have
$S/{\epsilon_S(H_J)+(z)}$ is a homomorphic image of $R/(x^{p-1}y)$. 
\end{discussionbox}

\begin{remarkbox}
\label{remark:embInherits}
Suppose that $\hilbertPoset_R$ admits an embedding $\epsilon$. Let 
$I = \epsilon(H)$ for some $H \in \hilbertPoset_R$. 
Then $\idealPoset_{R/I} \simeq \{J \in \idealPoset_R : I \subseteq J\}$ and 
$\hilbertPoset_{R/I} \simeq \{H_{J} : J \in \idealPoset_R, I \subseteq J\}$. 
In particular, $\epsilon$ induces an embedding of
$\hilbertPoset_{R/I}$ into
$\idealPoset_{R/I}$~\cite{CaKuEmbeddings10}*{Remark~2.8}.
Let $S = R[z]$. Then $\hilbertPoset_S$ admits an embedding $\epsilon_S$, by
Theorem~\ref{theorem:extnRings}. (See, also, the paragraph following the
theorem.) 
Let $J$ be an $S$-ideal that is in the image of $\epsilon_S$.
By above, $\hilbertPoset_{S/J}$ admits an embedding $\epsilon_{S/J}$.
Moreover suppose that $\epsilon_R$ is given by images of lex-segment ideals
of $A$, i.e., for all $H \in \hilbertPoset_R$, there is a lex-segment
$A$-ideal $L$ such that $\epsilon_R(H) = LR$. Then for all $H \in
\hilbertPoset_{S/J}$, there exists a lex-segment $A[z]$-ideal  $L$ such
that $\epsilon_{S/J}(H) = L(S/J)$; see~\cite
{CaKuEmbeddings10}*{Theorem~3.12 and Proposition~2.16}.
(In the lexicographic order on $A[z]$, $z$ is the last variable.)
We remark here that
in~\cite{MerminPeevaLexifying}*{Theorems~4.1 and~5.1}
Mermin and Peeva had shown that if $R$ is the quotient of
$A$ by a monomial ideal, then the lex-segment ideals of $A[z]$ give an
embedding of $\hilbertPoset_{S/J}$.
\end{remarkbox}

\begin{remarkbox}
\label{remark:embInheritsCL}
Let $A = \Bbbk[x_1, \ldots, x_n]$ and $B = A[x_{n+1}]$. Let 
$2 \leq e_1 \leq \cdots \leq e_{n+1} \leq \infty$.
Let $\fraka = (x_1^{e_1}, \ldots, x_{n}^{e_{n}})A$,
$\frakb = (x_1^{e_1}, \ldots, x_{n+1}^{e_{n+1}})B$, 
$R = A/\fraka$ and $S = B/\frakb$. Assume that we have
inductively constructed 
an embedding $\epsilon_R$ of $\hilbertPoset_R$ such that
for all $R$-ideals $I$, there exists a lex-segment
$A$-ideal $L$ such that $\epsilon_R(H_I)  = LR$.
Then $\hilbertPoset_S^{\mathrm{stab}}$
admits an embedding $\epsilon_S$, by Theorem~\ref{theorem:extnRings}.
We now argue that $\hilbertPoset_S^{\mathrm{stab}} = \hilbertPoset_S$.
Let $J$ be any $B$-ideal containing $\frakb$. Replacing $J$ by an initial
ideal, we may assume that $J$ is a monomial ideal.  
Therefore it suffices to prove that for all
monomial $B$-ideals $J$ containing $\frakb$, there is a $B$-ideal $J'$ such
that $J'S$ is $z$-stable and $H_{JS} = H_{J'S}$.
For this we may assume that $\Bbbk = \complex$.
Now, the remarks in the paragraph following Theorem~\ref{theorem:extnRings}
imply that $\hilbertPoset_S^{\mathrm{stab}} = \hilbertPoset_S$, so we have
an embedding $\epsilon_S$ of $\hilbertPoset_S$.
Again, by~\cite{CaKuEmbeddings10}*{Theorem~3.12 and Proposition~2.16} we
see that for all $S$-ideals $J$, there exists a lex-segment
$B$-ideal $L$ such that $\epsilon_S(H_J)  = LS$. As a corollary, we get the
theorem of Clements-Lindstr\"om mentioned in Section~\ref{sec:intro}. 
\end{remarkbox}

\section{Graded Betti numbers} 
\label{sec:gradedBettiNos}
We are now ready to state and prove the main result of this paper.
Let $\epsilon_R : \hilbertPoset_R \longrightarrow \idealPoset_R$ be an
embedding and $I$ be an $R$-ideal.
Then $\beta_{1,j}^R(R/I)
\leq \beta_{1,j}^R(R/\epsilon_R(H_{I}))$~\cite
{CaKuEmbeddings10}*{Remark~2.5}.
We do not know whether 
$\beta_{i,j}^R(R/I)
\leq
\beta_{i,j}^R(R/\epsilon_R(H_{I})) $ for all $i$
and $j$.
We show that if $
\beta_{i,j}^A(R/I)
\leq
\beta_{i,j}^A(R/\epsilon_R(H_{I}))
$ for all $i,j$, then a similar inequality
holds for the extension rings
considered in Theorem~\ref{theorem:extnRings}.
(In general, there are examples with
$ \beta_{i,j}^A(R/I) <
\beta_{i,j}^A(R/\epsilon_R(H_{I}))$~\cite{MeMuColoured10}*{Proposition~3.2}.)

\begin{theorem}
\label{theorem:bettiNumbers}
Let $t$ be a positive integer or $\infty$.
Let $S = R[z]/(z^t)$. 
Suppose that $\hilbertPoset_R$ admits an embedding $\epsilon_R$ and that
$\beta_{i,j}^A(R/I) \leq
\beta_{i,j}^A(R/\epsilon_R(H_{I}))$ for all
$R$-ideals $I$ and for all $i,j$.
Then for all $i, j$ and for all $z$-stable $S$-ideals $J$, 
$\beta_{i,j}^B(S/J) \leq \beta_{i,j}^B(S/\epsilon_S(H_{J}))$.
\end{theorem}

Suppose that $t$, $\Bbbk$ and $A_1$ satisfy the conditions 
discussed after Theorem~\ref{theorem:extnRings} that ensure that 
$\hilbertPoset_S^{\mathrm{stab}} = \hilbertPoset_S$. Then 
for all $S$-ideals
$J$, there exists a $z$-stable $S$-ideal $J'$ such that $H_{J'} = H_J$ and
$\beta_{i,j}^B(S/J) \leq \beta_{i,j}^B(S/J')$. 
This inequality of Betti numbers follows from the proofs
of~\cite{CaKuEmbeddings10}*{Lemmas 4.1 and~4.2}, and the
upper-semi-continuity of graded Betti numbers in flat families.
Hence, in these situations, we can conclude from
Theorem~\ref{theorem:bettiNumbers}
that 
for all $i$ and $j$ and for all $S$-ideals $J$, 
$ \beta_{i,j}^B(S/J) \leq \beta_{i,j}^B(S/\epsilon_S(H_{J}))$.

We begin with a (somewhat inefficient) bound on the Hilbert series of
Tor modules.
\begin{lemma}
\label{lemma:torAndHilbSer}
Let $R'$ be any positively graded $\Bbbk$-algebra of finite type.
Let $M$ and $N$ be finitely generated graded $R'$-modules. Then
$H_{\tor^{R'}_i(M,N)} \preccurlyeq H_M H_{\tor^{R'}_i(N,\Bbbk)}$. 
In particular, if $N$ is a graded ${R'}$-submodule of $M$ or a graded 
homomorphic image of $M$, then
$H_{\tor^{R'}_i(M,\Bbbk)}  \preccurlyeq
H_{\tor^{R'}_i(N,\Bbbk)}  + (H_M-H_N)  H_{\tor^{R'}_i(\Bbbk,\Bbbk)}$.
\end{lemma}

\begin{proof}
Let $F_\bullet$ be a minimal graded ${R'}$-free resolution of $N$. Then 
$H_{\tor^{R'}_i(M,N)} \preccurlyeq H_{M \otimes_{R'} F_i} = 
H_MH_{\tor^{R'}_i(N,\Bbbk)}$,
proving the first assertion. For the second assertion, 
let $L = \coker (N \to M)$ or $L = \ker (M \to N)$, as the case is. Then 
$H_{\tor^{R'}_i(M,\Bbbk)} 
\preccurlyeq H_{\tor^{R'}_i(N,\Bbbk)}  + H_{\tor^{R'}_i(L,\Bbbk)} \preccurlyeq
H_{\tor^{R'}_i(N,\Bbbk)}  + H_LH_{\tor^{R'}_i(\Bbbk,\Bbbk)}$ where the
first inequality follows from the exact sequence of $\tor$ and the
second one follows from the first part of this proposition. Now
note that $H_L = H_M - H_N$. 
\end{proof}

The following lemma is perhaps well-known to many readers, but we give a
proof for the sake of completeness.
\begin{lemma}
\label{lemma:TorRingChange}
Identify $A$ with the quotient ring $B/(zB)$. Then for all $B$-ideals
$\frakb$, 
$\tor_i^B(\frakb, \Bbbk) \simeq \tor_i^{A}(\frakb/z\frakb, \Bbbk)$.
\end{lemma}

\begin{proof}
Let $F_\bullet$ be a 
graded free $B$-resolution of $\frakb$. We can compute the Tor modules
$\tor_i^B(\frakb, B/(zB))$ either as the homology of 
$F_\bullet \otimes_B B/(zB)$ or as the homology of 
\[
(0 \to B(-1) \stackrel{z}{\to} B \to 0) \otimes_B \frakb \qquad =\qquad  
(0 \to \frakb(-1) \stackrel{z}{\to} \frakb \to 0).
\]
Notice that $z$ is a non-zerodivisor on $\frakb$, so 
$\tor_i^B(\frakb, B/(zB)) = 0$ for all $i > 0$, i.e, 
$F_\bullet \otimes_B B/(zB)$ is a graded $(B/(zB))$-free resolution of
$\frakb/z\frakb$. Now,
$\tor_i^B(\frakb, \Bbbk) 
= \homology_i(F_\bullet \otimes_B \Bbbk)  
= \homology_i((F_\bullet \otimes_B B/(zB)) \otimes_{B/(zB)} \Bbbk)  
= \tor_i^{B/(z)}(\frakb/z\frakb, \Bbbk) \simeq
\tor_i^{A}(\frakb/z\frakb, \Bbbk)$.
\end{proof}

\begin{proof}[Proof of Theorem~\ref{theorem:bettiNumbers}]
We will work with $B$-ideals. 
Let $J$ be a $B$-ideal containing $\mathfrak a B + (z^t)$ such that $JS$ is
$z$-stable.
Write $L$ for the preimage 
in $B$ of the $S$-ideal
$\epsilon_S(H_{JS})$. We need to show that
$ H_{\tor_{i}^B(J, \Bbbk)} \preccurlyeq H_{\tor_{i}^B(L, \Bbbk)} $.
By Lemma~\ref{lemma:TorRingChange},
it suffices to show that 
$ H_{\tor_{i}^A(J/zJ, \Bbbk)} \preccurlyeq H_{\tor_{i}^A(L/zL, \Bbbk)} 
$.
Write 
$J=\bigoplus_{i\in \naturals}  J_{\langle i \rangle}z^i$ and
$L=\bigoplus_{i\in \naturals}  L_{\langle i \rangle}z^i$ as $A$-modules.
Let $J'$ be the preimage of the $S$-ideal 
$\bigoplus_{i\in \naturals}  \epsilon_R(H_{J_{\langle i \rangle}R})z^i$.

Define graded $A$-modules 
\[
\begin{split}
M_1 & = \bigoplus_{j \in \frakI} (J_{\langle j \rangle}/J_{\langle {j-1} \rangle})(-j) \qquad\text{and}\\
M_2 & =
\begin{cases}
A/J_{\langle {t-1} \rangle}(-t), & \text{if}\; t \in \naturals, \\
0, & \text{if}\; t = \infty.
\end{cases}
\end{split}
\]
Then, as an $A$-module, $J/zJ = J_{\langle 0 \rangle} \oplus M_1 \oplus
M_2$. Similarly define $A$-modules $M'_1$ and $M'_2$ for $J'$ and
$N_1$ and $N_2$ for ${L}$. 
Notice that $\frakm_AM_1 = \frakm_AM'_1 = \frakm_AN_1 = 
0$, since $JS$, $J'S$ and ${LS}$ are $z$-stable, by hypothesis, by
Remark~\ref{remarkbox:EmbedIndividualPieces} and by 
Remark~\ref{remark:propertiesEX} respectively. (Here $\frakm_A$ is the
homogeneous maximal ideal of $A$.)
Therefore $M_1$, $M'_1$ and $N_1$ are (as $A$-modules) direct sums of
shifted copies of $\Bbbk$, so
\begin{equation}
\label{equation:TorOfVectorSpaces}
H_{\tor_{i}^A(M_1, \Bbbk)} = H_{M_1}H_{\tor^{A}_i(\Bbbk,\Bbbk)},\;
H_{\tor_{i}^A(M'_1, \Bbbk)} = H_{M'_1}H_{\tor^{A}_i(\Bbbk,\Bbbk)}
\;\text{and} \;
H_{\tor_{i}^A(N_1, \Bbbk)} = H_{N_1}H_{\tor^{A}_i(\Bbbk,\Bbbk)}.
\end{equation}
Moreover,
\begin{equation}
\label{equation:BettiTableDecompOneLessVar}
H_{\tor_{i}^A(J/zJ, \Bbbk)} = 
H_{\tor_{i}^A(J_{\langle 0 \rangle}, \Bbbk)} + 
H_{\tor_{i}^A(M_1, \Bbbk)} + H_{\tor_{i}^A(M_2, \Bbbk)}.
\end{equation}
Similar expressions exist for $J'$ and $L$.
By the hypothesis and the fact that $H_{M_1} = H_{M'_1}$, we see that 
$H_{\tor_{i}^A(J/zJ, \Bbbk)} \preccurlyeq H_{\tor_{i}^A(J'/zJ', \Bbbk)}$.
Therefore we may replace $J$ by $J'$ and assume that 
$J_{\langle i \rangle}R$ is in the image of $\epsilon_R$ for all
$i \in \naturals$.
Hence, by Theorem~\ref{theorem:extnRings}\eqref{theorem:strongHyp}, 
$L_{\langle 0 \rangle}  \subseteq J_{\langle 0 \rangle}$, and, 
if $t < \infty$ then
$J_{\langle t-1 \rangle}  \subseteq L_{\langle t-1 \rangle}$.
Hence, for all values of $t$, there is a surjective $A$-homomorphism $M_2
\to N_2$. 
Since $H_J =
H_{{L}}$, we see that
\begin{equation}
\label{equation:HSsplitsForIandIL}
H_{J_{\langle 0 \rangle}}+H_{M_1}+H_{M_2} = H_{J/zJ} = 
H_{L/zL} = 
H_{L_{\langle 0 \rangle}} + H_{N_1} + H_{N_2}.
\end{equation}
Therefore
\[
\begin{split}
H_{\tor_{i}^A(J/zJ, \Bbbk)} & = 
H_{\tor_{i}^A(J_{\langle 0 \rangle}, \Bbbk)} + 
H_{\tor_{i}^A(M_1, \Bbbk)} + 
H_{\tor_{i}^A(M_2, \Bbbk)} \\
& \preccurlyeq 
H_{\tor_{i}^A(L_{\langle 0 \rangle}, \Bbbk)} + 
H_{\tor_{i}^A(M_1, \Bbbk)} +
H_{\tor_{i}^A(M_2, \Bbbk)} +
(H_{J_{\langle 0 \rangle}}-H_{L_{\langle 0 \rangle}})
H_{\tor^{A}_i(\Bbbk,\Bbbk)}\\
& \preccurlyeq 
H_{\tor_{i}^A(L_{\langle 0 \rangle}, \Bbbk)} + 
H_{\tor_{i}^A(M_2, \Bbbk)} +
(H_{J_{\langle 0 \rangle}}-H_{L_{\langle 0 \rangle}}+H_{M_1})
H_{\tor^{A}_i(\Bbbk,\Bbbk)}\\
& =
H_{\tor_{i}^A(L_{\langle 0 \rangle}, \Bbbk)} + 
H_{\tor_{i}^A(M_2, \Bbbk)} +
(H_{N_1}+H_{N_2}-H_{M_2}) H_{\tor^{A}_i(\Bbbk,\Bbbk)}\\
& \preccurlyeq 
H_{\tor_{i}^A(L_{\langle 0 \rangle}, \Bbbk)} + 
H_{\tor_{i}^A(N_2, \Bbbk)} +
H_{N_1} H_{\tor^{A}_i(\Bbbk,\Bbbk)}\\
& = H_{\tor_{i}^A(L_{\langle 0 \rangle}, \Bbbk)} + 
H_{\tor_{i}^A(N_1, \Bbbk)} + H_{\tor_{i}^A(N_2, \Bbbk)}  \\
& = H_{\tor_{i}^A(L/zL, \Bbbk)}
\end{split}
\]
where the second line uses Lemma~\ref{lemma:torAndHilbSer},
the third line uses~\eqref{equation:TorOfVectorSpaces},
the fourth line uses~\eqref{equation:HSsplitsForIandIL},
the fifth line uses Lemma~\ref{lemma:torAndHilbSer}
and the next line 
uses~\eqref{equation:TorOfVectorSpaces}.
\end{proof}

The proof shows that, for fixed $i$ and $j$,
if $\beta_{i,j}^A(R/I) \leq \beta_{i,j}^A(R/\epsilon_R(H_{I}))$ 
and $\beta_{i-1,j}^A(R/I) \leq \beta_{i-1,j}^A(R/\epsilon_R(H_{I}))$ 
for all $R$-ideals $I$, then 
for all $z$-stable $S$-ideals $J$, 
$\beta_{i,j}^B(S/J) \leq \beta_{i,j}^B(S/\epsilon_S(H_{J}))$.

Now, as a corollary, we give a quick proof of a theorem of
Mermin-Peeva-Stillman and its generalization by Murai. 
First, we relabel $z$ as $x_{n+1}$ and
write $B = \Bbbk[x_1, \ldots, x_{n+1}]$.
We say that a monomial $B$-ideal
$J$ is  \define{strongly stable} if for all monomials $m \in I$
and for all $2 \leq j \leq n+1$  such that $x_j$ divides $m$, $x_i(m/x_j)
\in I$, for all $i \leq j$.

\begin{corollary}[\cite {MPSSquares08}*{Theorem~3.5}, 
\cite {MuraiBPP08}*{Theorem~1.1}]
\label{corollary:MPSMurai}
Let $J$ be a strongly stable monomial $B$-ideal. 
Let $\frakb = (x_1^{e_1}, \ldots, x_{n+1}^{e_{n+1}})$, with $2 \leq
e_1 \leq \cdots \leq e_{n+1} \leq \infty$.
Then 
\begin{asparaenum}
\item \label{enum:MPSMuraiBettiNosIndepChar}
For all $i,j$, the value of $\beta^B_{i,j}({J+\frakb})$ does not depend on
$\Bbbk$.
\item \label{enum:MPSMuraiBettiNoInequality}
For all $i,j$, 
$\beta^B_{i,j}({J+\frakb}) \leq \beta^B_{i,j}({L+\frakb})$ 
where $L$ is the lex-segment $B$-ideal such that 
$H_{L + \frakb} = H_{J+\frakb}$.
\end{asparaenum}
\end{corollary}

\begin{proof}
\underline{\eqref{enum:MPSMuraiBettiNosIndepChar}}: We may compute the
Betti numbers using Lemma~\ref{lemma:TorRingChange} 
and~\eqref{equation:BettiTableDecompOneLessVar}. Note that
$J_{\langle 0 \rangle}$ and 
$J_{\langle e_n-1 \rangle}$ are 
strongly-stable-plus-$(x_1^{e_1}, \ldots, x_n^{e_n})$. The assertion now
follows by induction.
\underline{\eqref{enum:MPSMuraiBettiNoInequality}}: 
From Remark~\ref{remark:embInheritsCL}, whose notation we adopt, we 
have embeddings $\epsilon_R$ and $\epsilon_S$ of $\hilbertPoset_R$ and
$\hilbertPoset_S^{\mathrm{stab}} = \hilbertPoset_S$
respectively. By induction on the number of variables, we see that the
hypothesis of Theorem~\ref{theorem:bettiNumbers} is satisfied. 
Now $\epsilon_S(H_{JS}) = LS$,
so Theorem~\ref{theorem:bettiNumbers} completes the proof.
\end{proof}

Starting with a ring whose poset of Hilbert functions admits an embedding,
it is possible to construct new examples of rings with the same property,
by using Theorem~\ref{theorem:extnRings} and
Remark~\ref{remark:embInherits} recursively. We thus recover the following
result of D.~A.~Shakin.

\begin{corollary}[\cite{ShakPieceLex03}*{Theorems~3.10 and~4.1}]
Let $\fraka_i$ be a lex-segment ideal in the ring
$\Bbbk[x_1,\dots, x_i]$, $i=1\leq n$, then let $\fraka = \sum_{i=1}^n
\fraka_iA$, with $A=\Bbbk[x_1,\dots, x_n]$. Let $R=A/\fraka$.
Then $\hilbertPoset_R$ admits an embedding $\epsilon_R$ induced by the
lexicographic order on $A$.
Moreover $\beta_{i,j}^A(R/I) \leq \beta_{i,j}^A(R/\epsilon_R(H_I))$,
for all $R$-ideals $I$ and for all $i,j$.
\end{corollary}

\subsection*{Consecutive cancellation in Betti tables}
We say that a Betti table $\beta^R(M)$ is
obtained by \define{consecutive cancellations} from $\beta^R(N)$ if there
exists a collection $\Lambda$ of triples $(i,j, n_{i,j}) \in \naturals
\times \ints \times \naturals$ such that
$\beta^R(N) = 
\beta^R(M) + \sum_{\Lambda} n_{i,j}(\bfe_{i,j} + \mathbf{e}_{i+1,j})$. 
(Note that $\mathbf{e}_{i,j} + \mathbf{e}_{i+1,j}$ is the Betti table of a
complex $R(-j) \xrightarrow{1} R(-j)$ 
concentrated in homological degrees $i+1$ and $i$.) Now revert 
to the
situation of Theorem~\ref{theorem:bettiNumbers}.  We have the following:

\begin{proposition}
\label{proposition:consecCancel}
If $\beta^A(R/I)$ can be obtained from $\beta^A(R/\epsilon_R(H_I))$ by
consecutive cancellations for all $R$-ideals $I$, then $\beta^B(S/J)$ can be
obtained from $\beta^B(S/\epsilon_S(H_J))$ by consecutive cancellations for
all $z$-stable $S$-ideals $J$.
\end{proposition}

\begin{proof}
As in the proof of Theorem~\ref{theorem:bettiNumbers}, which we follow
closely, we work with
$B$-ideals.
Let $J$ be a $B$-ideal containing $\mathfrak a B + (z^t)$ such that $JS$ is
$z$-stable.
Write $L$ for the preimage of the $S$-ideal
$\epsilon_S(H_{JS})$. 
Let $J'$ be the preimage of the $S$-ideal 
$\bigoplus_{i\in \naturals}  \epsilon_R(H_{J_{\langle i \rangle}R})z^i$,
as in the proof of the theorem.
We need to show that
$\beta^B(J)$ can be obtained from $\beta^B(L)$ by consecutive cancellations.
Note that, by our hypothesis, $\beta^B(J)$ can be obtained from
$\beta^B(J')$ by consecutive cancellations, since these Betti tables are
equal to $\beta^A(J/zJ)$ and $\beta^A(J'/zJ')$ respectively.
Therefore, we may replace $J$ by $J'$ and assume that 
$J_{\langle i \rangle}R$ is in the image of $\epsilon_R$ for all $i \in
\naturals$.
Then $L_{\langle 0 \rangle} \subseteq J_{\langle 0 \rangle}$, $N_2$ is a
graded homomorphic image of $M_2$ (for a morphism of degree zero) and
$H_{N_1} - H_{M_1} = 
H_{J_{\langle 0 \rangle}}-H_{L_{\langle 0 \rangle}}+H_{M_2}-H_{N_2}$. 
Hence we may place ourselves in the context of
Lemma~\ref{lemma:torAndHilbSer}. We have
\begin{align*}
\beta^B(L)-\beta^B(J)
& = \beta^A(L_{\langle 0\rangle})-\beta^A(J_{\langle 0\rangle}) +
   \beta^A(N_2)- \beta^A(M_2)+\beta^A(N_1)-\beta^A(M_1) \\
&= \beta^A(L_{\langle 0\rangle}) - \beta^A(J_{\langle 0\rangle}) +
   \beta^A(N_2) - \beta^A(M_2)+(H_{N_1}-H_{M_1})\beta^A(k) \\
& = [\beta^A(L_{\langle 0\rangle}) - \beta^A(J_{\langle 0\rangle}) + 
 (H_{J_{\langle 0\rangle}} - H_{L_{\langle 0\rangle}})\beta^A(k)]   \\
& \qquad + [\beta^A(N_2)-\beta^A(M_2)+(H_{M_2}-H_{N_2})\beta^A(k)].
\end{align*}
Lemma~\ref{lemma:SESConsecCanc} below,
now, completes the proof of the proposition.
\end{proof}

\begin{lemma}
\label{lemma:SESConsecCanc}
Let $A$ be a standard-graded polynomial ring. Let $M$ and $N$ be finitely
generated graded $A$-modules. If $N$ is a graded $A$-submodule of $M$ or
a graded homomorphic image of $M$, then $\beta^A(M)$ can be obtained from
$\beta^A(N) + (H_M-H_N) \beta^A(\Bbbk)$ by consecutive cancellations. Here,
for  a series $h=\sum_{i \in \naturals}h_i\hssymb^i$, 
we mean, by
$h\beta^A(\Bbbk)$, the (infinite) Betti table $\sum_{i \in \naturals} 
\beta^A(\Bbbk(-i)^{\oplus h_i})$.
\end{lemma}

\begin{proof}
We will prove this when $N$ is a graded submodule of $M$; the other case is
similar. From minimal graded $A$-free resolutions of $N$ and $M/N$, we can
construct a graded $A$-free resolution of $M$ that is not necessarily
minimal. Every non-minimal graded $A$-free resolution of $M$ is a direct
sum of a minimal graded $A$-free resolution of $M$  with copies of exact
complexes of the form $0 \to A(-j) \stackrel{1}{\to} A(-j) \to 0$ that is
concentrated in homological degrees $i$ and $i+1$~\cite
{eiscommalg}*{Theorem~20.2}; therefore 
$\beta^A(M)$ can be obtained from $\beta^A(N) + \beta^A(M/N)$ 
by consecutive cancellations. Since $H_{M/N} = H_M - H_N$, it suffices to
show that, after relabelling the modules, for any graded $A$-module $M$,
$\beta^A(M)$ can be obtained from $H_M\beta^A(\Bbbk)$ by consecutive
cancellations.

Let $(K_\bullet, \partial_\bullet)$ be a Koszul complex that is a minimal
graded $A$-free resolution of $\Bbbk$. Note that 
\[
\begin{split}
H_{\tor_i^A(M,\Bbbk)} & = 
H_{\ker(\partial_i \otimes_A M)} - 
H_{\image(\partial_{i+1} \otimes_A M)} \\
& = 
H_{K_i \otimes_R M} - 
H_{\image(\partial_i \otimes_A M)} - 
H_{\image(\partial_{i+1} \otimes_A M)}.
\end{split}
\]
The series $H_{\tor_i^A(M,\Bbbk)}, 0 \leq i \leq n$ determine $\beta^A(M)$.
Similarly $H_{K_i \otimes_R M}, 0 \leq i \leq n$ determine
$H_M\beta^A(\Bbbk)$.  The lemma now follows by noting that for each $i$,  
$H_{\image(\partial_i \otimes_A M)}$ is subtracted from 
$H_M\beta^A(\Bbbk)$ twice, at $i$ and at $i-1$.
\end{proof}

\subsection*{An Eliahou-Kervaire type resolution for $z$-stable ideals}
{When $I$ is a strongly stable $A$-ideal (see the paragraph above
Corollary~\ref{corollary:MPSMurai} for definition)}, a minimal
graded $A$-free resolution of $I$ is given by the Eliahou-Kervaire 
complex (see~\cite{ElKeMinimalReslns90}*{Theorem~2.1}), which can be
constructed as an iterated mapping cone for a
specific order on the set of minimal monomial generators of $I$. Iterated
mapping cones can always be used to construct free resolutions, but
they need not be minimal in general. When they give minimal resolutions,
they can be used to obtain exact expressions for Betti numbers. See, for
instance,~\cite {MuraiBPP08}, which uses an iterated mapping cone from~\cite
{MPSSquares08}. 
{This section does not use explicit results on embeddings, but
only the calculations in the proof of Theorem~\ref{theorem:bettiNumbers}.}
We first make an observation:
\begin{observationbox}
\label{observationbox:MC}
Let $\phi : M \to N$ be an injective map of finitely generated graded
$B$-modules and let $F_\bullet$ and $G_\bullet$ be minimal graded free
$B$-resolutions of $M$ and $N$ respectively. Denote the comparison map
$F_\bullet \to G_\bullet$ by $\Phi$. Then the mapping cone $C_\bullet$ of
$\Phi$ is a graded free $B$-resolution of $\coker \phi$. Moreover if
$\rank_B C_i = \dim_\Bbbk \tor_i^B(\coker \phi, \Bbbk)$, then $C_\bullet$
is a minimal resolution.
\end{observationbox}

Suppose that $J$ is a $z$-stable $S$-ideal. We give an interpretation of a
minimal graded $B$-free resolution of $S/J$ as an iterated mapping cone.
Replace $J$ by its preimage in $B$. For $i \in \frakI$, set 
$J_{\langle i \rangle} = (J :_B z^i) \cap A$. If $t$ is finite, then, for
all $i \geq t$, set $J_{\langle i \rangle} = J_{\langle t-1 \rangle}$.
Let $J' = \bigoplus_{i \in \naturals} J_{\langle i \rangle}z^i$. Note that
if $t$ is finite then $J = (J'+(z^t))$ while if 
if $t$ is infinite then $J = J'$. Note also that $J'R[z]$ is $z$-stable.

We first construct a minimal resolution of $J'$.
Let $f_1z^{j_1}, \ldots, f_{r}z^{j_{r}}, fz^j$ be a set of minimal
multigraded generators of $J'$ ordered such that $j_1 \leq \cdots \leq
j_r \leq j$.  We may assume 
that $j > 0$, for, otherwise, a minimal $B$-free
resolution of $J'$ can be obtained by applying $-\otimes_A B$ to a 
minimal $A$-free resolution of $(J' \cap A)$. Further, 
$J_{\langle i \rangle} = J_{\langle j \rangle}$ for all $i \geq j$.
Write $d=\deg f$.
Let $J'' = (f_1z^{j_1}, \ldots, f_{r}z^{j_{r}})$. Then 
for all $0 \leq i < j$, $(J'' :_B z^i) \cap A = J_{\langle i \rangle}$, and
for all $i \geq j$, $(J'' :_B z^i) \cap A = (J'' :_B z^j) \cap A$.
Moreover, 
$(J'' :_B fz^j) \supseteq (J_{\langle j-1 \rangle} :_A f)B = \frakm_AB$, so
$(J'' :_B fz^j) = \frakm_AB$.
Also, $J_{\langle j \rangle}/((J'' :_B z^j) \cap A) \simeq \Bbbk(-d)$.
There is a graded exact sequence
\[
0 \to \frac{B}{\frakm_AB}(-d-j) \xrightarrow{fz^j} \frac{B}{J''} \to \frac{B}{J'} \to 0
\]
of $B$-modules. Arguing as in the proof of
Theorem~\ref{theorem:bettiNumbers}, we see that 
$\beta^B(J') = \beta^A(J'/zJ') = 
\beta^A(J_{\langle 0 \rangle}) + \sum_{i \leq j}
\hssymb^iH_{\frac{J_{\langle i \rangle}}{J_{\langle i-1
\rangle}}}\beta^A(\Bbbk) = 
\beta^B(J'') + \beta^A(\Bbbk(-d-j))$. By
Observation~\ref{observationbox:MC},
the mapping cone
of a comparison morphism from a minimal $B$-free resolution of 
$\frac{B}{\frakm_AB}(-d-j)$ to that of  $\frac{B}{J''}$ gives a minimal
$B$-free resolution of $\frac{B}{J'}$. 

If $t$ is infinite, then we have
by now constructed a minimal $B$-free resolution of $S/J$ as an iterated
mapping cone.

Now suppose that $t$ is finite. Then, as we noted earlier, $J = (J'+(z^t))$.
Observe that $j < t$ and hence that $(J':_B z^t) = 
J_{\langle t-1 \rangle}B$. We have a graded exact sequence
\[
0 \to \frac{B}{J_{\langle t-1 \rangle}B}(-t) \xrightarrow{z^t} 
\frac{B}{J'} \to \frac{B}{J} 
\to 0
\]
of $B$-modules. As we saw above,
$\beta^B(J') = \beta^A(J_{\langle 0 \rangle}) + \sum_{1 \leq i < t}
\hssymb^iH_{\frac{J_{\langle i \rangle}}{J_{\langle i-1
\rangle}}}\beta^A(\Bbbk)$. Since 
$\beta^B(\frac{B}{J_{\langle t-1 \rangle}B}(-t)) = 
\beta^A(\frac{A}{J_{\langle t-1 \rangle}}(-t))$, we see
from~\eqref{equation:BettiTableDecompOneLessVar} that $\beta^B(J) = 
\beta^B(J') + \beta^B(\frac{B}{J_{\langle t-1 \rangle}B}(-t))$. Again,
by Observation~\ref{observationbox:MC},
the mapping cone for the comparison map of the minimal
$B$-free resolutions of 
$\frac{B}{J_{\langle t-1 \rangle}B}(-t)$ and $\frac{B}{J'}$ gives a minimal
$B$-free resolution of $B/J$.

\section{The lex-plus-powers conjecture reduces to the Artinian
situation}

In this section, we discuss some examples of applications of
Theorem~\ref{theorem:bettiNumbers}. We begin with a recursive application
of the result. Then we use the theorem to reduce the lex-plus-powers
conjecture to the Artinian situation.

Let $\bff = f_1, \ldots, f_c$ be a homogeneous $A$-regular
sequence (where $A = \Bbbk[x_1, \ldots, x_n]$) of degrees $e_1 \leq \cdots \leq
e_c$. Let $\fraka$ be the ideal generated by $\bff$ and let $\frakb$ be the
ideal generated by $x_1^{e_1},\dots, x_c^{e_c}$. We say $\mathbf f$
\define{satisfies the Eisenbud-Green-Harris conjecture} 
if there exists an inclusion of posets
$\hilbertPoset_{A/\fraka}\subseteq \hilbertPoset_{A/\frakb}$. 
(Since $\hilbertPoset_{A/\frakb}$ admits an embedding induced by the
lexicographic order on $A$ by the Clements--{Lindstr\"om} theorem, this
definition is equivalent to the seemingly stronger condition that 
for each $H\in \hilbertPoset_{A/\fraka}$ there exists a lex-segment
$A$-ideal $L$ such that $H_{LA/\frakb} = H$.) It is known that the
Eisenbud-Green-Harris conjecture reduces to the Artinian case.
More precisely, after a linear change of
coordinates (by replacing $\Bbbk$ by a suitable extension field, if
necessary), we may assume that $f_1, \ldots, f_c, x_{c+1}, \ldots, x_{n}$
is a maximal regular sequence. Now, write
$\overline{\bff} = \overline{f}_1, \ldots, \overline{f}_c$ for the image of
$\bff$ after going modulo $x_{c+1}, \ldots, x_{n}$. Then 
$\overline{\bff}$ is a maximal regular sequence in 
$\Bbbk[x_1, \ldots, x_c]$. If $\overline{\bff}$ satisfies the
Eisenbud-Green-Harris conjecture, then so does $\bff$~\cite
{CaMaEGH08}*{Proposition~10}.

The lex-plus-powers conjecture of Evans (see~\cite {FrRiLPP07})
asserts that if $\bff$ satisfies the Eisenbud-Green-Harris conjecture, $I$
is an $A$-ideal containing $\bff$ and $L$ is the lex-segment $A$-ideal
such that $H_I = H_{LA/\frakb}$,  then $\beta_{i,j}^A(I) \leq
\beta_{i,j}^A(L+\frakb)$ for all $i,j$. 
We show now that, similarly, the lex-plus-powers conjecture reduces to the
Artinian case. We keep the notation from above, and, further
denote $L+\frakb$ by $\lpp_{\bff}(I)$.

\begin{theorem}
\label{theorem:lppReducesToArtinian}
If $\overline{\bff}$ satisfies the Eisenbud-Green-Harris and the
lex-plus-powers conjectures, then so does $\bff$.
\end{theorem}

\begin{proof}
By~\cite
{CaMaEGH08}*{Proposition~10}, 
$\bff$ satisfies the Eisenbud-Green-Harris conjecture. Hence we need to
show that if $J$ is an $A$-ideal containing $\bff$, then 
$H_{\tor_i^A(J, \Bbbk)} \preccurlyeq H_{\tor_i^A(\lpp_{\bff}(J), \Bbbk)}$. 
We may assume that $c<n$. 
Let $\tilde{\bff} = \tilde {f_1},\dots,\tilde {f_c}$ the images
of $\bff$ in $\tilde{A} : =\Bbbk[x_1,\dots,x_{n-1}]$.
Inductively we can assume that for all 
$\tilde{A}$-ideals $I$ containing $\tilde{\bff}$, 
$H_{\tor_i^{\tilde{A}}(I, \Bbbk)} \preccurlyeq 
H_{\tor_i^{\tilde{A}}(\lpp_{\tilde{\bff}}(I), \Bbbk)}$.
Let $J$ be an $A$-ideal containing $\bff$. 
By taking the initial ideal with respect to a weight $w$, 
$w(x_1) = \cdots = w(x_{n-1}) = 1$ 
and $w(x_n) = 0$, we may assume that $J$ contains
$\tilde{\bff}$ and that it has a decomposition (as an $\tilde{A}$-submodule
of $A$) $J = \bigoplus_i J_{<i>}x_n^i$ where the $J_{<i>}$ are
$\tilde{A}$-ideals containing $\tilde{\bff}$. By applying~\cite
{CaKuEmbeddings10}*{Lemma~4.1} 
we may assume that $J(A/\fraka)$ is $x_n$-stable.
The proof of this lemma involves taking initial ideals and applying
distraction.
The values of $\beta_{i,j}^A(-)$ do not decrease
{while}
taking initial ideals (which can be argued, e.g., the same way as in~\cite
{eiscommalg}*{Theorem~15.17}), and remain unchanged
{while applying} distraction
(since distraction can be interpreted as polarization followed by going
modulo a regular element~\cite{CaSbLPPCoh12}*{Section~1.3}).
Let ${J'} = \bigoplus_i \lpp_{\tilde{\bff}}(J_{<i>})x_n^i$, where, for each
$i$, $\lpp_{\tilde{\bff}}(J_{<i>})$ is the lex-plus-powers ideal of
$\tilde{A}$ for the sequence $\tilde{\bff}$. Note that we can obtain the
ideals $\lpp_{\tilde{\bff}}(J_{<i>}) \tilde{A}/\frakb$ by 
embedding their Hilbert functions, so ${J'}A/\frakb$ is $x_n$-stable. 
We now claim that for all $i$, 
$H_{\tor_i^{A}({J'}, \Bbbk)} \succcurlyeq H_{\tor_i^{A}(J, \Bbbk)}$. 
Assume the claim. Now $\lpp_{\bff}(J)$ is the preimage of
$\epsilon_{A/\frakb}(H_{J'})$, so by
Theorem~\ref{theorem:bettiNumbers}, 
$H_{\tor_i^{A}(\lpp_{\bff}(J), \Bbbk)} \succcurlyeq H_{\tor_i^{A}(J, \Bbbk)}$. 

Now to prove the claim, we follow the strategy of the proof of
Theorem~\ref{theorem:bettiNumbers}. Note that since ${J'}$ and $J$ are
$x_n$-stable, 
\[
{J'}/x_n{J'} = \lpp_{\tilde{\bff}}(J_{<0>}) \oplus M \;\text{and} \;
J/x_nJ = J_{<0>} \oplus N 
\]
for some $\tilde{A}$-modules $M$ and $N$ that are annihilated by 
$(x_1, \ldots, x_{n-1})$. Since $H_{\lpp_{\tilde{\bff}}(J_{<0>})} = 
H_{J_{<0>}}$, we see that $H_M = H_N$. By the induction hypothesis, 
$H_{\tor_i^{\tilde{A}}(\lpp_{\tilde{\bff}}(J_{<0>}), \Bbbk)} \succcurlyeq 
H_{\tor_i^{\tilde{A}}(J_{<0>}, \Bbbk)}$. Hence
$H_{\tor_i^{A}({J'}, \Bbbk)} \succcurlyeq H_{\tor_i^{A}(J, \Bbbk)}$. 
\end{proof}


\def\cfudot#1{\ifmmode\setbox7\hbox{$\accent"5E#1$}\else
  \setbox7\hbox{\accent"5E#1}\penalty 10000\relax\fi\raise 1\ht7
  \hbox{\raise.1ex\hbox to 1\wd7{\hss.\hss}}\penalty 10000 \hskip-1\wd7\penalty
  10000\box7}

\end{document}